\documentclass[]{article}

\usepackage{amsmath,amssymb,amsthm}
\usepackage{arcs}
\usepackage[section]{algorithm}
\usepackage{algorithmicx}
\usepackage{algpseudocode}
\usepackage{comment}
\usepackage{pstricks}
\usepackage{pst-plot}
\usepackage{frcursive}
\usepackage[section]{placeins}

\usepackage[]{graphicx}

\newtheorem{theorem}{Theorem}[section]
\newtheorem{lemma}[theorem]{Lemma}

\newtheorem{proposition}[theorem]{Proposition}
\newtheorem{corollary}[theorem]{Corollary}

\numberwithin{equation}{section}

\begin{document}
\title{The pivotal set of a Boolean function}


\author{
Rapha\"el Cerf\footnote{
\noindent 
Universit\'e Paris-Saclay, CNRS, Laboratoire de math\'ematiques d'Orsay, 91405, Orsay.}
}




\maketitle
\begin{abstract}
We define the pivotal set of a Boolean function
and we prove a fundamental inequality on its expected size,
when the inputs are independent random coins of parameter~$p$.
We give two complete proofs of this inequality. 
Along the way,
we obtain the classical Margulis--Russo formula. 
We give a short proof
of the classical Hoeffding inequality for i.i.d. Bernoulli
random variables, and we use it to derive more complex 
deviations inequalities
associated to the pivotal set.
We follow finally Talagrand's footsteps and we discuss
a beautiful inequality that he proved in the uniform case.
\end{abstract}



\def\zk{\{\,1,\dots,k\,\}}
\def\zu{\{0,1\}}
\def\ud{\{\,1,\dots,d\,\}}
\def\unn{\{\,1,\dots,n\,\}}
\def\zun{\{\,0,\dots,n\,\}}
\def\omegab{\overline{\omega}}
\def\rcurs{\text{\begin{cursive}rl\end{cursive}}}
 \def \Z {{\mathbb Z}}
 \def \E {{\mathbb E}}
 \def \Zd {{\mathbb Z}^d}
 \def \Sd {S^{d-1}}
\def \wQ {\widetilde{Q}}
 \def \R {{\mathbb R}}
 \def \Rd {{\mathbb R}^d}
 \def \C {{\mathbb C}}
 \def \cD {{\mathcal D}}
 \def \cE {{\mathcal E}}
 \def \cV {{\mathcal V}}
 \def \cH {{\mathcal H}}
 \def \cR {{\mathcal R}}
 \def \cF {{\mathcal F}}
 \def \cC {{\mathcal C}}
 \def \cN {{\mathcal N}}
 \def \cT {{\mathcal T}}
 \def \cP {{\mathcal P}}
 \def \cPl {{\mathcal P}_\ell}
 \def \cS {{\mathcal S}}
 \def \tS {\smash{\widetilde S}}
 \def \bC {{\overline C}}
 \def \bE {{\overline E}}
 \def \bA {{\overline A}}
 \def \bB {{\overline B}}
 \def \ta {\text{two--arms}}
 \def \N {{\mathbb N}}
 \def \P {{\mathbb P}}
 \def \E {{\mathbb E}}

\newcommand{\set}[2]{\big\{\, #1 :  #2\, \big\}}
\newcommand{\setscolon}[2]{\big\{ #1\, ; \, #2 \big\}}
\newcommand{\Set}[2]{\Big\{\, #1\, : \, #2 \,\Big\}}

\newcommand{\bd}{\partial\, } 
\newcommand{\din}{\partial^{\, in}}
\newcommand{\dout}{\partial^{\, out}}
\newcommand{\dini}{\partial^{\, in}_\infty}
\newcommand{\douti}{\partial^{\, out}_*}
\newcommand{\dine}{\partial^{\, in}_{ext}}
\newcommand{\doutie}{\partial^{\, out,ext}_{\infty}}
\newcommand{\doute}{\partial^{\, out,ext}}
\newcommand{\dedge}{\partial^{\, edge}}
\newcommand{\dexte}{\partial^{ext,\, edge}}
\newcommand{\dstar}{\partial^{\, *}}
\newcommand{\dcirc}{\partial^{\, o}}
\newcommand{\diam}{\text{diameter}\,}
\newcommand{\shell}{\text{Shell}\,}

\newcommand{\ii}{{\underline{i}}} 
\newcommand{\uu}{{\underline{u}}} 
\newcommand{\jj}{{\underline{j}}} 
\newcommand{\kk}{{\underline{k}}} 
\renewcommand{\ll}{{\underline{l}}} 
\newcommand{\mm}{{\underline{m}}} 
\newcommand{\nn}{{\underline{n}}} 
\newcommand{\xx}{{\underline{x}}} 
\newcommand{\zz}{{\underline{z}}} 
\newcommand{\yy}{{\underline{y}}} 
\newcommand{\oo}{{\underline{0}}} 
\newcommand{\rr}{{\underline{r}}} 
\newcommand{\iinf}{{\underline{$\infty$}}} 

\renewcommand{\AA}{\underline{A}}
\newcommand{\EE}{\underline{E}}
\newcommand{\OO}{\underline{O}}
\newcommand{\BB}{\underline{B}}
\newcommand{\DD}{\underline{D}}
\newcommand{\CC}{\underline{C}}
\newcommand{\cCC}{\underline{\calC}}
\newcommand{\FF}{\underline{F}}
\newcommand{\HH}{\underline{H}}
\newcommand{\Su}{\underline{S}} 
\newcommand{\MM}{\underline{M}}  
\newcommand{\RR}{\underline{R}}

\newcommand{\tQ}{\smash{\widetilde{Q}}}
\newcommand{\La}{\Lambda}
\newcommand{\Lak}{\La(k)} 
\newcommand{\Lan}{\La(n)}
\newcommand{\Lam}{\La(m)}
\newcommand{\LaN}{\La(N)}
\newcommand{\LaM}{\La(M)}
\newcommand{\Lann}{{\La(n-\sqrt n)}} 
\newcommand{\Lanf}{\La(n,\phi)}
\newcommand{\LLanf}{\underline{\La}(n,\phi)}
\newcommand{\LLan}{\underline{\La}(n)}
\newcommand{\LLam}{\underline{\La}(m)}
\newcommand{\LLa}{\underline{\La}}
\newcommand{\M}{{\mathbb{M}}} 
\newcommand{\T}{{\mathbb{T}}} 
\newcommand{\Td}{{\mathbb{T}^d}} 
\newcommand{\Eb}{{\mathbb{E}}} 
\renewcommand{\L}{{\mathbb{L}}} 
\renewcommand{\H}{{\mathbb{H}}} 
\newcommand{\Zdn}{{\mathbb{Z}}^d_n} 
\newcommand{\ZZd}{{\underline{\mathbb{Z}}^d}} 
\newcommand{\oomega}{{\underline{\omega}}} 
\newcommand{\ze}{\underline{0}}
\newcommand{\Hd}{{\mathbb{H}}^d}
\newcommand{\Dd}{{\mathbb{D}}^d} 
\newcommand{\Nd}{{\mathbb{N}}^d}
\newcommand{\Ld}{{\mathbb{L}}^d}
\newcommand{\Ldi}{\smash{{\mathbb{L}}^{d,\infty}}}
\newcommand{\Edi}{{\mathbb{E}}^{d,\infty}}
\newcommand{\Ldstar}{\cal{L}^{d,*}} 
\newcommand{\Ed}{{\mathbb{E}}^d} 
\newcommand{\Edstar}{\cal{E}^{d,*}} 
\newcommand{\Eds}{{\mathbb{E}}^{d,*}}

\newcommand{\dist}{{\rm dist}} 
\newcommand{\distT}{\, {\rm dist}_{\calF}} 
\newcommand{\distH}{\, {\rm dist}_{\calH}} 
\newcommand{\DistP}{\, {\rm dist}} 
\newcommand{\dinf}{ {\rm d}_{\infty} }
\newcommand{\di}{{\rm d}_\infty\, }
\newcommand{\done}{\, {\rm d}_{1}\, }
\newcommand{\dtwo}{\, {\rm d}_{2}\, }

\section{Introduction}
\label{MRF}
A boolean function is a function which takes only
two values $0$ and $1$, it transforms a very complex
input into a binary result. Consider for instance
a function 
of an image which is meant to 
answer a complex question, like:

\noindent
$\bullet$ Is there a trumpet somewhere in the image?

\noindent
Or for more serious and important applications: 

\noindent
$\bullet$ Airport safety:
Based on the
scanner image, should a luggage be
checked?

\noindent
$\bullet$ Medicine:
Is there a suspicion
of cancer in the MRI image?

\noindent
While the first question can be easily answered by anyone, the second can only be answered properly by qualified baggage screeners, and the third by skilled doctors. Today, neural networks are trained on huge databases using deep learning techniques and they achieve excellent results on 
these problems, which we only dreamed of twenty years ago.
In each case,
the outcome
of the learning process can be seen as a boolean function
of the pictures, albeit a very complex one, which depends
on billions of parameters.
A related problem is to guess 
the value of a boolean function when only 
partial information is available. The difficulty
is to understand which part of the input of the 
boolean function plays a crucial role to decide 
its output.
The
analysis of boolean functions
has grown into
a very active field of research, 
due to its multiple applications
in several branches of science (see for instance
the book of O'Donnell \cite{odonnell}).
Our goal here is to introduce the notion of the pivotal
set, which is indeed
pivotal in the understanding of a boolean function. We state and prove in full detail a basic and
classical inequality controlling the expected size of the
pivotal set. 
This proof is not new, 
it is certainly known to experts in the field, however 
we feel it worth presenting and discussing the full details,
all the more that 
introductory texts on Boolean functions typically
rely on tools from the Fourier analysis to derive it.
We shall then develop deviations
estimates associated to this inequality, which
will gradually lead us to a beautiful inequality
due to Talagrand. This inequality is a part of a vast
circle of ideas originating in harmonic analysis, which 
Talagrand developed
to obtain several revolutionary inequalities. We will mention
only two of these: Talagrand's famous
concentration inequality and the isoperimetric inequality
on the hypercube \cite{IT}. 
The inequality presented here relies on
some concentration estimates, but it goes 
in the opposite direction to the isoperimetric inequality.

The project to revisit some intermediate 
results of Talagrand, with an adaptation to the biased case,
started two years ago.
The recent attribution of the Abel prize to Talagrand
prompted me to conclude this project, 
indeed this is an opportunity to explain an important contribution
of Talagrand which is less known than his other accomplishments.
However, before embarking into this program,
we need to introduce some notation and to precise
the model we shall study.

\noindent
{\bf The pivotal set.}
Let $n\geq 1$ and let $f$ be a boolean function defined
on
$\Omega\,=\,\{\,0,1\,\}^n$, i.e.,
$$f:
\omega=\big(\omega(1)\cdots\omega(n)\big)\mapsto f(\omega)\in\{\,0,1\,\}\,.$$
Given a configuration $\omega$, we say that the $i$--th coordinate $\omega(i)$
of $\omega$ is pivotal if the value of $f(\omega)$ is decided by $\omega(i)$, 
or equivalently, if changing the value of $\omega(i)$ changes the value of $f$.
More precisely,
for $\omega\in\Omega$ 
and $i\in\unn$, we denote by $\omega^i$
(respectively $\omega_i$) the configuration $\omega$ where the $i$--th component is set
to $1$ (respectively to $0$), that is,
\begin{equation*}
	\omega^i(j)\,=\,
	\begin{cases}
		\omega(j)&\quad \text{ if }j\neq i\\
		1&\quad \text{ if }j= i\\
	\end{cases}\,,\qquad
	\omega_i(j)\,=\,
	\begin{cases}
		\omega(j)&\quad \text{ if }j\neq i\\
		0&\quad \text{ if }j= i\\
	\end{cases}\,.
\end{equation*}
The 
$i$--th coordinate $\omega(i)$
is said to be pivotal for the function $f$ in the configuration $\omega$
if $f(\omega^i)\neq f(\omega_i)$.
We denote by $\cP(f,\omega)$ 
the set of the pivotal coordinates of $\omega$ for $f$.
In general, little can be said on the pivotal set, because it depends
in a complicated way on both $\omega$ and $f$.
So we will focus on the specific situation where $f$ is non--decreasing, meaning that
$$\forall \omega\in\Omega\quad
	\forall i\in\unn\qquad
f(\omega_i)\leq f(\omega^i)\,,
	$$
and the configuration $\omega$ is chosen randomly.
The pivotal set $\cP(f,\omega)$ becomes de facto a random subset
of $\unn$.

\noindent
{\bf Randomness.}
To go further, we must specify the random distribution of $\omega$.
We fix a parameter $p$ in $[0,1]$ and, to decide the status of
each component $\omega(i)$, we draw a coin of parameter $p$, the
coins being independent. For $p=\frac{1}{2}$, this corresponds to
the uniform distribution over $\Omega$.

\noindent
{\bf Influences.}
The probability that the $i$--th coordinate 
belongs to $\cP(f,\omega)$ is called the influence of the coordinate
$i$. It provides a mean of quantifying the 
importance of one specific coordinate to the output of the boolean
function, and also to compare the relative importance of two distinct
coordinates. 
The sum of the influences of all the coordinates is
called the total influence, or the averaged sensitivity. 

So, why is it important to control the pivotal set of a boolean function?
Friedgut's junta theorem tells that, whenever the total influence of
a function $f$ is small, the function $f$ can be approximated
by a junta, that is a function depending on a small number 
of coordinates. With the help of this result, O'Donnell and 
Servidio \cite{dose}
developed an algorithm able to learn a monotone function $f$
to any constant accuracy, 
in time polynomial in $n$ and in the decision tree size of $f$.
On a more general level,
the links between the structure of a boolean function $f$ 
and the probabilistic properties of its pivotal set are complex and still mysterious. A famous theorem of 
Benjamini, Kalai and Schramm \cite{BKS} states that, 
if the sum of the squares of the influences of a function is small, then
it is noise sensitive: even if two inputs are highly correlated, the
outputs of the function might be highly non--correlated.
The 
noise sensitivity theorem relies on the deep inequality due to
Talagrand that we shall present at the end of this text.

We come now to the first mathematical result presented here.
A simple exchange between summation and expectation 
shows that the total influence is in fact
equal to the expected size of the pivotal set.
It turns out that, surprisingly, 
we can compute an explicit upper bound on 
the expected size of the pivotal set.
\begin{theorem}
	\label{maine}
	For any non--decreasing boolean function $f$,
	we have
	\begin{equation}
		\label{bpf}
		E\big(|\cP(f)|\big)\,\leq\,
		\sqrt{\frac{n E(f) }{p(1-p)}}\,.
	\end{equation}
\end{theorem}
\noindent
Naturally,
the symbol $E$ is the expectation with respect to the probability $P$
on $\Omega$ defined by
\begin{equation*}
	\forall \omega\in\Omega\qquad
	P(\omega)\,=\,\prod_{1\leq i\leq n}p^{\omega(i)}(1-p)^{1-\omega(i)}\,,
\end{equation*}
and we remove the variable $\omega$ in the expectations, for instance
\begin{equation*}
	E(f) \,=\,
	\sum_{\omega\in\Omega}f(\omega)
	P(\omega)\,.
\end{equation*}
To the best of our knowledge, this result was
proved for the
first time by Friedgut and Kalai (see lemma~$6.1$ in \cite{FK}).
Benjamini, Kalai and Schramm quote this result (theorem~3.3 in 
\cite{BKS}) and they mentioned that it is
a consequence of \cite{CCFS} (although it does not appear explicitly
there).
\smallskip

\noindent
{\bf Application to the majority function.}
Suppose that the configuration $\omega$ represents the votes in an election, 
and that
the votes are independent and follow a Bernoulli distribution. The function
$f$ is taken to be the majority function, which is equal to $1$ if there is a
majority of votes equal to $1$ and $0$ otherwise. In this model, all the
votes play the same role, so that
$E\big(|\cP(f)|\big)=nP(1\text{ is pivotal})$
and theorem~\ref{maine} yields that
\begin{equation}
	\label{upbo}
	P(1\text{ is pivotal})
	\,\leq\,
	\frac{1 }{\sqrt{np(1-p)}}\,,
\end{equation}
hence the probability that a fixed vote plays a decisive role goes to $0$.
Moreover, in the uniform case $p=1/2$, 
the inequality~\eqref{upbo} captures the correct order, indeed it can be proved
that
\begin{equation*}
	P(1\text{ is pivotal})
	\,\sim\,
	\sqrt{\frac{2 }{\pi n}}
	\qquad \text{as}\qquad n\to\infty\,.
\end{equation*}

\smallskip

\noindent
However the constant in inequality~\eqref{bpf} is not optimal.
In fact, the original argument 
of Friedgut and Kalai \cite{FK} consisted in proving that the majority
function maximizes $E\big(|\cP(f)|\big)$ over all the boolean
functions~$f$.


Our first major goal is to provide the most elementary 
proof of theorem~\ref{maine}. 
The key idea consists in expressing the probability that a site
$i$ is pivotal with the help of a discrete derivative 
operator $\delta_i$.
As a by product, this proof yields 
another important formula due
to Margulis and Russo. We shall next work to obtain a better control
on the pivotal set, with a technique involving the Parseval formula
in Fourier analysis. 
It
is the strategy employed in modern textbooks
on boolean functions to prove
theorem~\ref{maine} 
(see theorem~$2.33$ of \cite{odonnell}), in fact 
it yields a stronger result. 
In order to understand better the link between pivotal coordinates
and the value of $f$, we shall compute the conditional
expectation of the size of the pivotal set knowing the value
of the function.
Our second major goal is to obtain more complex 
deviations inequalities
associated to the pivotal set.
The key tool here is 
the classical Hoeffding inequality
(to make the text self--contained,
we give a short proof for the case of 
i.i.d. Bernoulli
random variables).
In the end, we extend 
to the 
Bernoulli product
measure of parameter~$p$ an inequality that Talagrand proved
in the case $p=1/2$.
We follow finally Talagrand's footsteps and we discuss
a beautiful inequality that he proved in the uniform case
$p=1/2$.
Talagrand's inequality is a far--reaching extension 
of theorem~\ref{maine}. 
The starting point is the deviations 
inequality that is developed in the second major goal
of the text. 
The very statement and proof of the inequality testify to Talagrand's virtuosity.
The noise sensitivity theorem of 
Benjamini, Kalai and Schramm \cite{BKS}
relies in a crucial way upon a further generalization
of this inequality.

\nopagebreak
\section{The discrete derivative $\delta_i$}
The first step to prove the inequality~\eqref{bpf} is
to express 
		$|\cP(f,\omega)|$
with the help
of the discrete derivative of $f$ along
the $i$--th component, which is defined as
\goodbreak
\begin{equation}
	\forall\omega\in\Omega\qquad
	\delta_i f(\omega)\,=\,
	f(\omega^i)- f(\omega_i)\,. 
\end{equation}
Indeed, the
$i$--th coordinate $\omega(i)$ of $\omega$ is pivotal if
and only if 
	$\delta_i f(\omega)=1$, therefore
	\begin{equation}
		\label{inii}
		|\cP(f,\omega)|\,=\,
		\sum_{i=1}^n
	\delta_i f(\omega)\,.
	\end{equation}
Taking the expectation and using the linearity, we obtain
\begin{equation}
	E\big(|\cP(f,\omega)|\big)\,=\,
		E\Big(\sum_{i=1}^n \delta_i f(\omega)\Big)
		\,=\,
		\sum_{i=1}^n E\big(\delta_i f(\omega)\big)\,.
\end{equation}
We have now to evaluate
\begin{equation}
	\label{toev}
	E\big(\delta_i f(\omega)\big)
	\,=\,
	E\big(f(\omega^i)\big)
	-E\big(f(\omega_i)\big)
	\,.
\end{equation}
To that end, we will rely on a little trick,
recorded in the next lemma.
The simplicity of the calculation should not make us 
underestimate its power.
\begin{lemma} 
	\label{trle}
	Let $f$ be an arbitrary function defined on $\Omega$ with values
	in $\R$. For any $i\in\unn$, we have
	\begin{equation}
		\label{tric}
		\sum_{\omega\in\Omega}f(\omega)\,
 {\omega(i)}
		\, P(\omega)
		\,=\,
		\sum_{\omega\in\Omega}
	p\,f(\omega^i)
		\, P(\omega)
		\,.
	\end{equation}
\end{lemma}
\begin{proof}
The identity~\eqref{tric} can be verified in several ways. 
Our presentation is a bit heavy, but it does not make appeal to any new notation.
So, we write
	\begin{multline}
		\label{fubi}
		\sum_{\omega\in\Omega}f(\omega)\, {\omega(i)}
		\, P(\omega)
		\,=\,
		\kern-12pt
		\sum_{\omega(1),\cdots,\omega(n)\in\zu}
		\kern-12pt
		f\big(\omega(1)\cdots\omega(n)\big) \,{\omega(i)}
		\kern-1.4pt
	\,\prod_{1\leq j\leq n}
		\kern-2.3pt
	p^{\omega(j)}(1-p)^{1-\omega(j)}\\
		\,=\,
	\sum_{\omega(1)}
	\cdots
	\sum_{\omega(i-1)}
	\sum_{\omega(i+1)}
	\cdots
	\sum_{\omega(n)}
		f\big(\omega(1)\cdots\omega(i-1)\,1\,\omega(i+1)\cdots\omega(n)\big) 
		\cr
		\hfill\times
		p
	\,\prod_{
	\genfrac{}{}{0pt}{2}{1\leq j\leq n}{j\neq i}
	}p^{\omega(j)}(1-p)^{1-\omega(j)}\\
		\,=\,
	\sum_{\omega(1)}
	\cdots
	\sum_{\omega(i-1)}
	\sum_{\omega(i+1)}
	\cdots
	\sum_{\omega(n)}
		f\big(\omega^i\big) \, P(\omega^i)\,.
	\end{multline}
Now the summand 
$f\big(\omega^i\big) \, P(\omega^i)$ does not depend any more
on $\omega(i)$, 
so we 
reintroduce artificially the summation over $\omega(i)$ by writing
\begin{equation}
	\label{arti}
		f\big(\omega^i\big) \, P(\omega^i)\,=\,
	\sum_{\omega(i)}
		f\big(\omega^i\big) 
		\,p\, P(\omega)\,.
		\end{equation}
We plug the equation~\eqref{arti} into \eqref{fubi} and we obtain
the identity~\eqref{tric}.
\end{proof}
\noindent
The next corollary presents the analogous formula dealing with $1-\omega(i)$ 
instead of $\omega(i)$. 
Of course, this formula could be obtained with a direct computation, but we will
derive it from lemma~\ref{trle}.
\begin{corollary} 
\label{reeh}
For any function
 $f$ from $\Omega$ to
$\R$ and for any $i\in\unn$, we have
	\begin{equation}
		\label{trid}
		\sum_{\omega\in\Omega}f(\omega)\,
\big(1- {\omega(i)}\big)
		\, P(\omega)
		\,=\,
		\sum_{\omega\in\Omega}
		(1-p)\,f(\omega_i)
		\, P(\omega)\,.
\end{equation} 
\end{corollary} 
\begin{proof}
	For $\omega\in\Omega$, we denote by 
	$\omegab$ the configuration obtained by flipping all the components of $\omega$:
\begin{equation*}
	\forall j\in\unn\qquad
	\omegab(j)\,=\,
	1-\omega(j)\,.
\end{equation*}
Of course, the distribution of $\omega$ under $P$ is the same as the distribution of 
$\omegab$ under $P_{1-p}$, hence
\begin{align}
	\label{tarid}
	\sum_{\omega\in\Omega}f(\omega)\,
	\big(1- {\omega(i)}\big)
	\, P(\omega)
	\nonumber
	&\,=\,
	\sum_{\omega\in\Omega}f(\omegab)\,
	\big(1- {\omegab(i)}\big)
	\, P_{1-p}(\omega)\\
	&\,=\,
	\sum_{\omega\in\Omega}f(\omegab)\,
	{\omega(i)}
	\, P_{1-p}(\omega)
	\,.
\end{align}
We apply lemma~\ref{trle} to the function $\omega\mapsto
	f(\omegab)$ and the probability measure 
	$P_{1-p}$
and we get
\begin{multline}
	\label{tarie}
	\sum_{\omega\in\Omega}f(\omegab)\,
	{\omega(i)}
	\, P_{1-p}(\omega)
	\,=\, \sum_{\omega\in\Omega}
		(1-p)\,
		f\big(\overline{\omega^i}\big)
		\, P_{1-p}(\omega)
\cr
	\,=\, (1-p)
	\sum_{\omega\in\Omega}
		f(\omegab_i)
		\, P_{1-p}(\omega)
	\,=\, (1-p)
	\sum_{\omega\in\Omega}
		f\big({\omega_i}\big)
		\, P_{p}(\omega)
		\,.
\end{multline}
Putting together
formulas~\eqref{tarid} and~\eqref{tarie}, we obtain~\eqref{trid}.
\end{proof}
\noindent
Taking advantage of lemma~\ref{trle} 
and corollary~\ref{reeh}, we have
\begin{equation}
	\label{adva}
	E\big(f(\omega^i)\big)
	-E\big(f(\omega_i)\big)
	\,=\,
	\sum_{\omega\in\Omega}f(\omega)
	\Big(\frac{\omega(i)}{p}- \frac{1-\omega(i)}{1-p}
	\Big)
	\, P(\omega)
	\,.
\end{equation}
\section{Computation of $E\big(|\cP(f)|\big)$}
The righthand side of formula~\eqref{toev} has been computed in 
formula~\eqref{adva}.
In order to rewrite this last formula in a more concise form,
	we introduce the random variable $X_i$ defined by
\begin{equation}
	\label{exprX}
\forall\omega\in\Omega\qquad
X_i(\omega)\,=\, \frac{\omega(i)}{p}- \frac{1-\omega(i)}{1-p}\,,\end{equation}
and, substituting~\eqref{adva} into~\eqref{toev}, we obtain
\begin{equation} 
		\label{modif}
		E\big(fX_i\big)\,=\,
		E\big(\delta_i f\big)
		\,.
\end{equation} 
In the next step,
we sum this formula over $i\in\unn$ and we get
\begin{equation} 
		\label{smodif}
		\sum_{1\leq i\leq n}
		E\big(fX_i\big)\,=\,
		\sum_{1\leq i\leq n}
		E\big(\delta_i f\big)
		\,.
\end{equation}
In the final step, we put the sum inside the expectation,
we combine the above formula with the initial identity~\eqref{inii}
and we get the formula stated in the next proposition.
\begin{proposition}
	\label{snf}
	For any  non--decreasing boolean function $f$,
	we have
\begin{equation}
		\label{finii}
		E\big(
		|\cP(f)|
		\big)\,=\,
		E\big(
		S_n
	f\big)\,,
	\end{equation}
where $S_n$ is the random sum 
$S_n\,=\,X_1+\cdots+X_n $.
\end{proposition}
\noindent
In fact, lemma~\ref{trle},
corollary~\ref{reeh} and proposition~\ref{snf} can be viewed
as simple consequences of the following elementary formulas.
Let $\omega(1)$ be a Bernoulli random variable with parameter~$p$:
\[
	P(\omega(1)=0)=1-p\,,\qquad
	P(\omega(1)=1)=p\,.
\]
Let $f$ be a function from $\{\,0,1\,\}$ to $\R$. We have
\begin{equation}
	\label{elef}
	f(1)-f(0)
	\,=\,
	E\Bigg(
		f(\omega(1))
\Big(\frac{\omega(1)}{p}- \frac{1-\omega(1)}{1-p}\Big)
\Bigg)\,.
		\end{equation}
To prove directly formula~\eqref{modif}, we compute
		$E\big(fX_i\big)$ with the help of
		Fubini's theorem. More precisely, 
		we fix the variables
$\omega(1),\dots,\omega(i-1),\omega(i+1),\dots,\omega(n)$ and we compute
		first the expectation with respect to
		$\omega(i)$. This amounts to take the conditional expectation of $f$ given
$\omega(1),\dots,\omega(i-1),\omega(i+1),\dots,\omega(n)$. 
From formula~\eqref{elef}, we have
\begin{multline}
	\label{condex}
	E\big(fX_i\,|\,
	\omega(1),\dots,\omega(i-1),\omega(i+1),\dots,\omega(n)
	\big)
	\,=\,\cr
	\delta_if\big(
	\omega(1),\dots,\omega(i-1),\cdot,\omega(i+1),\dots,\omega(n)
	\big)\,.
\end{multline}
Naturally, the resulting function does not depend any more on 
the variable $\omega(i)$. Taking the expectation of formula~\eqref{condex}, we obtain
formula~\eqref{modif}. To be frank, this is just another presentation of the computation done in the 
proof of lemma~\ref{trle}, which might look more natural for
readers acquainted with the conditional expectation, but which
is not as elementary.

\section{Completion of the proof of theorem~\ref{maine}}
We start from the formula of
proposition~\ref{snf}: 
\begin{equation}
		\label{ifinii}
		E\big( |\cP(f)| \big)\,=\,
		E\big( S_n f\big)\,=\,
		\sum_{\omega\in\Omega}S_n(\omega)f(\omega)P(\omega)\,.
\end{equation}
A straightforward application of 
the Cauchy--Schwarz inequality gives
	\begin{equation*}
		\label{mcs}
		\Big|
		\sum_{\omega\in\Omega}S_n(\omega)f(\omega)P(\omega)
		\Big|
		\,\leq\,
		\sqrt{
			\Big(
		\sum_{\omega\in\Omega}\big(S_n(\omega)\big)^2 P(\omega)
		\Big)
		\Big(
		\sum_{\omega\in\Omega}\big(f(\omega)\big)^2 P(\omega)
		\Big)
	}
		\,.
	\end{equation*}
This inequality can be rewritten in a more concise form with the help of
the expectation $E$ as
\begin{equation}
	\label{cso}
	\big| E\big( S_n f\big)\big|\,\leq\,
	\sqrt{
		E\big((S_n)^2\big)
	E\big((f)^2\big)}\,.
\end{equation}
The random variable $S_n$ is the sum of $n$ i.i.d. centered variables,
so that
	\begin{equation}
		\label{mr3f}
		E\big((S_n)^2\big)\,=\,
\sum_{1\leq i\leq n} E\big((X_i)^2\big)\,=\,
\frac{n}{p(1-p)}\,.
	\end{equation}
Putting together
\eqref{ifinii},
\eqref{cso}
and \eqref{mr3f}, we obtain the inequality
stated in theorem~\ref{maine}.
\section{The Margulis--Russo formula}
We state here an important consequence of proposition~\ref{snf}.
\begin{proposition}
	For any  non--decreasing boolean function $f$,
	we have
	\begin{equation}
		\label{mr1}
		\frac{d}{dp}E(f)\,=\,
		E\big(|\cP(f)|\big)\,.
	\end{equation}
\end{proposition}
\noindent
The Margulis--Russo formula~\eqref{mr1} is very convenient
	to study the response of $E(f)$ to a slight
	variation of the parameter $p$, for instance to analyze threshold phenomena.
Once we have proposition~\ref{snf}, the proof is straightforward.
We take the derivative of $E(f)$ with respect to the parameter $p$:
\begin{equation*}
		\frac{d}{dp}E(f)
		\,=\,
		\frac{d}{dp}
		\sum_{\omega\in\Omega}f(\omega)\,P(\omega)
		\,=\,
		\sum_{\omega\in\Omega}f(\omega)
		\, 
		\frac{d}{dp}
		P(\omega)\,.
\end{equation*}
We differentiate $P(\omega)$ either as an $n$--fold product, or through 
its logarithmic derivative. Indeed, we have
\begin{equation*}
	\forall \omega\in\Omega\qquad
	\ln P(\omega)\,=\,
	\sum_{1\leq i\leq n}\Big(
		{\omega(i)}\ln p+
	({1-\omega(i)})\ln(1-p)\Big)\,.
\end{equation*}
Whichever way, we get,
for any $\omega\in\Omega$, 
\begin{equation}
	\label{deriv}
		\frac{d}{dp}P(\omega)
		\,=\,
	\sum_{1\leq i\leq n}
		X_i(\omega)
		\, P(\omega)
		\,=\,
		S_n(\omega)
		\, P(\omega)
		\,,
\end{equation}
whence
\begin{equation*}
		\frac{d}{dp}E(f)
		\,=\,
		E\big(fS_n\big)
		\,,
\end{equation*}
and formula~\eqref{mr1} follows immediately from
proposition~\ref{snf}.
\section{An alternative proof via Bessel's inequality}
Techniques coming from the Fourier analysis play a central role
in the study of boolean functions. 
We will show here how a stronger result than
theorem~\ref{maine}
can be derived from 
the classical Bessel inequality.
	We use the random variable $X_i$ defined in~\eqref{exprX} and 
	the formula~\eqref{modif}
	to write
\begin{equation} 
		\label{rmodif}
P\big(i \text{ is pivotal for $f$}\big)\,=\,
		E\big(fX_i\big)
		\,.
\end{equation}
We endow the space 
${\mathcal F}(\Omega,\R)$ 
of the functions from $\Omega$ to $\R$ with the
scalar product
\begin{equation*}
	(f,g)\in{\mathcal F}(\Omega,\R)\mapsto
	\langle f,g\rangle\,=\,
	\sum_{\omega\in\Omega}
	f(\omega)g(\omega) P(\omega)\,.
\end{equation*}
We denote by $||\cdot||_2$ the associated norm on
${\mathcal F}(\Omega,\R)$. 
This turns
${\mathcal F}(\Omega,\R)$ into a Hilbert space.
Moreover the $n$ random variables 
$X_1,\dots,X_n$ 
form an orthogonal family, more precisely we have
\begin{equation}
	\label{norm}
	\forall i,j\in\unn\qquad
	\langle X_i,X_j\rangle\,=\,
	\begin{cases}
	\phantom{p(1}\,0& \text{ if }i\neq j\,,\cr
\displaystyle\frac{1}{p(1-p)}
		& \text{ if }i= j\,.
	\end{cases}
\end{equation}
From this perspective, 
		the quantity $E\big(fX_i\big)$ is simply the
		scalar product
	$\langle f,X_i\rangle$ and a direct application of Bessel's inequality
	yields that
\begin{equation}
	\label{parse}
	\sum_{1\leq i\leq n}
\frac{\langle f,X_i\rangle^2}{\langle X_i,X_i\rangle^2}
	\,\leq\,
	\big(||f||_2\big)^2\,=\,
	E(f)
\,.
	\end{equation}
Combining the identities~\eqref{rmodif}, the
little computation~\eqref{norm} 
and inequality~\eqref{parse},
	we obtain the inequality stated in the next proposition.
\begin{proposition}
	For any  non--decreasing boolean function $f$,
	we have
	\begin{equation}
		\label{gr1}
	\sum_{1\leq i\leq n}
 \Big(P\big(i \text{ is pivotal for $f$}\big)\Big)^2\,\leq\,
		\frac{ E(f) }{p(1-p)}\,.
	\end{equation}
\end{proposition}
\noindent
This last inequality~\eqref{gr1} is in fact stronger than the inequality~\eqref{bpf}.
Indeed, it follows from 
the Cauchy--Schwarz inequality that
	\begin{equation*}
		\label{icfp}
		E\big(|\cP(f)|\big)
		\,=\,
	\sum_{1\leq i\leq n}
P\big(i \text{ is pivotal for $f$}\big)
\,\leq\,
\sqrt{n
	\sum_{1\leq i\leq n}
 \Big(P\big(i \text{ is pivotal for $f$}\big)\Big)^2}\,,
	\end{equation*}
and using~\eqref{gr1} to bound the last term, we recover 
the ine\-qua\-li\-ty~\eqref{bpf}.
\section{The conditional expectation 
of $|\cP(f)|$}
\label{seco}
We will be interested in controlling the cardinality of the 
pivotal set for $f$
whenever the function $f$ is equal to~$1$.
Our first step in this direction consists in derivating a formula
involving the conditional expectation
		$E\big(|\cP(f)|\,\big|\,f=1\big)$. 
We focus again
on the specific situation where $f$ is non--decreasing. 
In this case,
whenever a coordinate $i$ is pivotal for $f$,
the value of $f$ is equal to $\omega(i)$.
Moreover
		the status of the coordinate $i$ itself is independent of the fact that $i$ is pivotal
		or not, therefore
		\begin{equation}
			P\big(i\in \cP(f),\,f=1\big)\,=\,
			P\big(i\in \cP(f),\, \omega(i)=1\big)
					      \,=\,
			      P\big(i\in \cP(f)\big)\,p\,,
	\end{equation}
which we rewrite as
\begin{equation}
	\label{rw1}
      P\big(\,i\in \cP(f)\,\big)\,=\,
      \frac{1}{p}
			P\big(i\in \cP(f),\,f=1\big)\,.
\end{equation}
Summing equation~\eqref{rw1} over $i$, and putting the summation
inside the expectation, we obtain
\begin{equation}
\label{mr7}
		E\big(|\cP(f)|\big)
		\,=\,
      \frac{1}{p}
      E\big(|\cP(f)|\,{{f}}\big)
		\,.
	\end{equation}
	The identities~\eqref{finii} and~\eqref{mr7} yield
\begin{equation}
\label{mr8}
		E\big(S_nf\big)
		\,=\,
      \frac{1}{p}
      E\big(|\cP(f)|\,{{f}}\big)
		\,.
	\end{equation}
Dividing finally by $P(f=1)$, we obtain the formula stated in the next lemma.
\begin{lemma} 
	\label{ral}
	For any  non--decreasing boolean function $f$,
	we have
	\begin{equation}
\label{mru}
		E\big(S_n\,\big|\,f=1\big)
		\,=\,
      \frac{1}{p}
		E\big(|\cP(f)|\,\big|\,f=1\big)
		\,.
	\end{equation}
\end{lemma}
\noindent
The previous formula~\eqref{mru} is very interesting. Indeed, it tells us that,
whenever the boolean function $f$ takes the value $1$, the expected size
of the pivotal set coincides with the conditional expectation of $S_n$ knowing that $f=1$. 
Up to an affine rescaling,
the sum $S_n$ is a
binomial random variable and probabilists know very well how to control the
binomial distribution in various regimes.
Ultimately, we would like to link the size of the pivotal set 
		$\cP(f)$ to the probability~$P(f=1)$.
Our strategy to do so is the following. Suppose that 
the righthand side of~\eqref{mru} is large.
This forces
that the conditional law of $S_n$ knowing that $f=1$ is concentrated on subsets
of $\Omega$ where $S_n$ is large, but these sets have small probability,
so for the ratio 
${E(S_nf)}/{P(f=1)}$ 
to be large, 
the event $f=1$ must also have
a small
probability.
\section{Hoeffding's inequality}
The key tool to complete the previous program is the classical Hoeffding inequality~\cite{HOE}.
The general version of the inequality deals with sums of independent bounded variables.
We restate next the inequality for the specific case which concerns us, namely the
sum $S_n$.
Recall that $S_n$
is the random sum 
$S_n\,=\,X_1+\cdots+X_n $, where the variables $X_1,\dots,X_n$ are i.i.d. with
distribution given by
$$\forall i\in\unn\qquad 
P\Big(X_i=\frac{1}{p}\Big)\,=\,p\,,\qquad
P\Big(X_i=-\frac{1}{1-p}\Big)\,=\,1-p\,.$$
We provide a detailed proof of Hoeffding's inequality in this case.
The proof is short (one page) and elementary, but it is quite tricky. 
\begin{proposition} For any $n\geq 1$ and any $p\in [0,1]$, we have
\begin{equation}
	\label{hoef}
	\forall u>0\qquad P\big(|S_n|\geq u\big)\,\leq\,2\exp\Big(-
	\frac{{2p^2(1-p)^2}{}u^2}{n}
\Big)\,.
\end{equation}
\end{proposition}
\begin{proof}
Let $u>0$ be fixed and let $\lambda>0$ be an additional parameter, to be chosen later.
We write
\begin{multline} P\big(S_n\geq u\big)\,=\,
	\label{mark}
	P\big(\lambda S_n\geq\lambda u\big)
	\,
	=\,
	P\big(\exp\big(\lambda S_n\big)\geq\exp(\lambda u)\big)\cr
	\,\leq\,
	\exp(-\lambda u)
	E\big(\exp\big(\lambda S_n\big)\big)\,,
\end{multline}
where we have applied the Markov inequality in the last step.
Using the fact that the variables $X_i$ are i.i.d., we compute then
\begin{equation}
	\label{compu}
	E\big(\exp\big(\lambda S_n\big)\big)\,=\,
	\prod_{1\leq i\leq n}E\big(\exp\big(\lambda X_i\big)\big)
		\,=\,\phi(\lambda)^n
		\,,
	\end{equation}
where we have set $\phi(\lambda)=
E\big(\exp\big(\lambda X_1\big)\big)$.
The explicit expression of $\phi(\lambda)$ is
\begin{equation*}
\phi(\lambda)\,=\,
	E\big(\exp\big(\lambda X_1\big)\big)\,=\,
		p \exp\big(\lambda /p\big)\,+\,
		(1-p) \exp\big(-\lambda /(1-p)\big)\,.
\end{equation*}
Using the convexity of the exponential, we see that 
$\phi(\lambda)\,\geq\,1$ for any $\lambda\in\mathbb R$.
We take advantage of this inequality to perform a little symmetrization
trick:
for any $\lambda>0$, we have
\begin{equation}
	\label{crux}
\phi(\lambda)\,\leq\,
\phi(\lambda) \phi(-\lambda)
\,=\,p^2+
		2p (1-p) 
		\cosh\Big(\frac{\lambda }{p(1-p)}\Big)
		+(1-p)^2	\,.
\end{equation}
The function $\cosh$ 
is even and satisfies 
$1\leq\cosh(x)\leq\exp(x^2/2)$, 
so from the inequality~\eqref{crux}
we deduce that
\begin{equation}
\phi(\lambda)\,\leq\,
		\cosh\Big(\frac{\lambda }{p(1-p)}\Big)
\,\leq\,
		\exp\Big(\frac{\lambda^2 }{2p^2(1-p)^2}\Big)\,.
\end{equation}
Substituting this inequality back in~\eqref{compu} and~\eqref{mark}, 
we obtain
\begin{equation} 
	P\big(S_n\geq u\big)\,\leq\,
		\exp\Big(-\lambda u+\frac{n\lambda^2 }{2p^2(1-p)^2}\Big)\,.
\end{equation} 
It remains to choose the optimal value for $\lambda$. Obviously,
we obtain the best inequality 
with $\lambda= {up^2(1-p)^2}/n$, i.e.,
\begin{equation} 
	\label{subopt}
\forall u >0\qquad
	P\big(S_n\geq u\big)\,\leq\,
		\exp\Big(-\frac{p^2(1-p)^2u^2 }{2n}\Big)\,.
\end{equation}
		The same inequality holds for
	$P\big(S_n\leq -u\big)$ (it suffices to change $p$ into
	$1-p$ and to apply the previous inequality to $-S_n$).
	The sum of these two inequalities
	yields the inequality~\eqref{hoef}
	and creates the factor~$2$ there. 
	To be honest, our proof
yields a suboptimal constant in the exponential: the $2$
comes in the denominator in formula~\eqref{subopt}. 
A little refinement in the
above chain of inequalities would give 
	the inequality~\eqref{hoef} with the $2$
	in the numerator.
\end{proof}
\section{An exponential inequality}
With the Hoeffding inequality in the hand, we shall complete the program described at the
end of section~\ref{seco} and we shall prove the following general exponential inequality,
which does not even require that the boolean function~$f$ is monotone.
\begin{proposition}
	\label{bth}
	For any boolean function $f$ defined on 
	$\{\,0,1\,\}^n$,
	we have
	\begin{equation}
		\label{vyin}
	\forall p\in [0,1]\qquad
	 P(f=1) \,\leq\,
2\exp\Big(
-\frac{ 1}{2n}
	\Big( 
		p(1-p) 
		E\big(S_n\,\big|\,f=1\big)
\Big)^2\Big)\,.
	\end{equation}
\end{proposition}
\begin{proof}
	We start with
\begin{equation}
\label{ci}
		E\big(|S_n|\,\big|\,f=1\big)
		\,=\,
		\frac{1}{P(f=1)}
		\int_{\{\,f=1\,\}}
		|S_n|\,dP\,.
	\end{equation}
	Using Fubini's theorem, 
we rewrite the integral in~\eqref{ci} as follows:
\begin{equation*}
	\label{wboo}
		\int_{\{\,f=1\,\}}
	|S_n|\,dP
	\,=\,
		\int_\Omega 
	\bigg(
	\int_{0}^{+\infty}1_{u\leq |S_n|}\,du
	\bigg)f
	\,dP
	\,=\,
	\int_{0}^{+\infty}
	\kern-7pt
	P\big(|S_n|\geq u,f=1\big)\,
	du\,.
\end{equation*}
Let $t>0$ and let us split this integral in two:
\begin{align}
	\label{wbfo}
	\int_{0}^{+\infty}
	\kern-7pt
	P\big(|S_n|\geq u,f=1\big)\,
	du
	&\,=\,
	\int_{0}^{t}
	\cdots
	+
	\int_{t}^{+\infty}
	\kern-7pt
	\cdots\cr
	&\,\leq\,
	tP(f=1)+
	\int_{t}^{+\infty}
	\kern-7pt
	P\big(|S_n|\geq u\big)\,
	du\,.
\end{align}
Reporting~\eqref{wbfo} in~\eqref{ci}, we have
\begin{equation}
\label{cin}
		E\big(|S_n|\,\big|\,f=1\big)
		\,\leq\,
      \frac{1}{P(f=1)}
	\int_{t}^{+\infty}
	P\big(|S_n|\geq u\big)\,
	du\,+\,t\,.
\end{equation}
Now $S_n$ is the sum of $n$ i.i.d. centered random variables $X_1,\dots,X_n$
satisfying
\begin{equation*}
\forall i\in\unn\qquad -\frac{1}{1-p}\,\leq\,X_i\,\leq\,\frac{1}{p}\,.\end{equation*}
Substituting
Hoeffding's inequality~\eqref{hoef}
into~\eqref{cin}, we obtain
\begin{equation}
\label{gci}
		E\big(|S_n|\,\big|\,f=1\big)
		\,\leq\,
      \frac{2}{P(f=1)}
	\int_{t}^{+\infty}
	\exp\Big(- \frac{{2p^2(1-p)^2}{}u^2}{n}\Big)
	\,du\,+\,t\,.
	\end{equation}
We make the change of variable 
$ 2p(1-p)u= \sqrt{n}v $, so that~\eqref{gci} becomes
\begin{equation}
\label{cgci}
		E\big(|S_n|\,\big|\,f=1\big)
		\,\leq\,
		\frac{\sqrt{n}}{p(1-p)P(f=1)}
	\int_{
		\frac{2p(1-p)t}{\sqrt{n}}
	}^{+\infty}
	\exp\Big(- \frac{v^2}{2}\Big)
	\,dv\,+\,t\,.
	\end{equation}
Recalling that this inequality holds for any $t>0$,
we replace $t$ by 
$\frac{\sqrt{n}t}{2p(1-p)}$ and we get
\begin{equation}
\label{hci}
		E\big(|S_n|\,\big|\,f=1\big)
		\,\leq\,
\frac{\sqrt{n}}{p(1-p)}
		\bigg(
			\frac{1}{P(f=1)}
	\int_{t}^{+\infty}
	\exp\Big(- \frac{v^2}{2}\Big)
\,dv\,+\,\frac{t}{2}\bigg)\,.
	\end{equation}
The righthand quantity admits a unique global minimum at
the point 
	$t^*=\sqrt{ {2} \ln ({2}/{P(f=1)}) }$
and we conclude that
\begin{equation}
\label{fci}
		E\big(|S_n|\,\big|\,f=1\big)
		\,\leq\,
\frac{1 }{p(1-p)}
	\sqrt{ {2n} \ln \frac{2}{P(f=1)} }
		\,.
	\end{equation}
We rewrite this inequality in exponential form as
\begin{equation}
\label{zfci}
	 P(f=1) \,\leq\,
2\exp\Big(
-\frac{ p^2(1-p)^2}{2n}
	\Big( 
		E\big(|S_n|\,\big|\,f=1\big)
\Big)^2\Big)\,.
	\end{equation}
We use finally the fact that
	\begin{equation}
		\Big| E\big(S_n\,\big|\,f=1\big) \Big|
		\,\leq\,
		 E\big(|S_n|\,\big|\,f=1\big)\,,
	\end{equation}
to obtain the inequality~\eqref{vyin} of the proposition.
\end{proof}
\noindent
In the case of a 
non--decreasing boolean function $f$, 
thanks to lemma~\ref{ral},
we can rewrite 
		$pE\big(S_n\,\big|\,f=1\big)$ as
		$E\big(|\cP(f)|\,\big|\,f=1\big)$ and we get 
		immediately the following corollary.
\begin{corollary} 
\label{imme}
	For any 
non--decreasing boolean function $f$ 
	defined on 
	$\{\,0,1\,\}^n$, we have
\begin{equation}
		\label{myin}
	\forall p\in [0,1]\qquad
	 P(f=1) \,\leq\,
2\exp\Big(
-\frac{ 1}{2n}
	\Big( 
		(1-p) 
		E\big(|\cP(f)|\,\big|\,f=1\big)
\Big)^2\Big)\,.
\end{equation}
\end{corollary} 
\noindent
This is a deviations inequality: it tells us that, if the conditional expectation of the cardinality
of the pivotal set is larger than $\sqrt{n}$, then it becomes extremely unlikely that
the function $f$ takes the value $1$.
\section{Following Talagrand's footsteps}
In fact, the technique employed in the proof of proposition~\ref{bth}
to bound the expectation
		$E\big(|S_n|\,\big|\,A\big)$ has been
		used routinely in various contexts
		by Talagrand. It is presented in his
most recent book (see lemma~2.3.2 in \cite{UBT}). 
Notably, Talagrand used this technique to derive 
a beautiful inequality
on the correlation on increasing sets in his influential work
\cite{INCT}. However, only the very beginning of the proof 
of the correlation inequality involves this technique,
i.e., the first step of the proof of Proposition~$2.2$
in~\cite{INCT}. This kind of upper bound is such a routine
for Talagrand that the detailed proof of 
our proposition~\ref{bth}
corresponds essentially
to $6$ lines in~\cite{INCT}!
So, what is the point of the previous computations?
A major difference with~\cite{INCT} is that we deal here with the
Bernoulli product measure of parameter~$p$, 
whereas Talagrand was considering the uniform case $p=1/2$.
So the subgaussian inequality of the uniform case is replaced
here by Hoeffding's inequality.
Another difference is that the constants in Talagrand's inequality are not explicit
(Talagrand did not care about their values),
whereas we do provide explicit values here.
Let us proceed a bit further along the computations of Talagrand.
We implement now the second step of 
Proposition~$2.2$ in~\cite{INCT} in 
our biased context and we see where it leads.
	We first improve
	proposition~\ref{bth} to get the following 
	deviations inequality.
\begin{proposition}
	\label{rth}
	For any  boolean function $f$,
	we have
	\begin{equation}
		\label{tmyin}
	 P(f=1)\,\leq\,
2\exp\Big(
-\frac{ 1}{2}
	\sum_{1\leq i\leq n}
	\Big(
		p(1-p)	
	E\big(X_if \,\big|\,f=1\big)
\Big)^2\Big)\,.
\end{equation}
\end{proposition}
\begin{proof}
	The proof follows the same strategy than the proof
	of proposition~\ref{bth}, the difference being
	that,
	instead of $S_n$,
	we work with the random variable 
$\tS_n$ given by
\begin{equation}
	\label{tsn}
	\tS_n\,=\,
	\sum_{1\leq i\leq n}
	\alpha_i X_i\,,
\end{equation}
where the $\alpha_i$'s are non--negative real numbers 
such that
	$\sum_{1\leq i\leq n}(\alpha_i)^2=1$.
Let $t>0$. Formula~\eqref{cin} holds for $\tS_n$ as well:
\begin{equation}
\label{tcin}
		E\big(|\tS_n|\,\big|\,f=1\big)
		\,\leq\,
      \frac{1}{P(f=1)}
	\int_{t}^{+\infty}
	P\big(|\tS_n|\geq u\big)\,
	du\,+\,t\,.
\end{equation}
Now $\tS_n$ is the sum of $n$ i.i.d. centered random variables 
which satisfy
\begin{equation}
\label{rboxu}
\forall i\in\unn\qquad -\frac{\alpha_i}{1-p}\,\leq\,\alpha_iX_i\,\leq\,\frac{\alpha_i}{p}\,.\end{equation}
Applying
the more general
Hoeffding's inequality \cite{HOE} to $\tS_n$, we get
\begin{equation}
	\label{hoefb}
	\forall u>0\qquad P\big(|\tS_n|\geq u\big)\,\leq\,
	2\exp\Big(-
	{{2p^2(1-p)^2}{}u^2}
\Big)\,.
\end{equation}
We perform then exactly the same step as in the 
proof of proposition~\ref{bth}, except that the
factor ${n}$ is replaced by $1$.
This way we obtain the following inequality:
\begin{equation}
\label{tfci}
		E\big(|\tS_n|\,\big|\,f=1\big)
		\,\leq\,
\frac{1 }{p(1-p)}
	\sqrt{ {2} 
	\ln \frac{2}{P(f=1)} }
		\,.
	\end{equation}
We use the inequality
$		\big| E\big(\tS_n\,\big|\,f=1\big) \big|
		\leq
		 E\big(|\tS_n|\,\big|\,f=1\big)$ and
		 we rewrite the inequality~\eqref{tfci} 
		 in exponential form to get
\begin{equation}
\label{tzfci}
	 P(f=1) \,\leq\,
2\exp\Big(
-\frac{1}{2} p^2(1-p)^2
	\Big( 
		E\big(\tS_n\,\big|\,f=1\big)
\Big)^2\Big)\,.
	\end{equation}
	Using the linearity of the expectation, 
	we have
\begin{equation*}
	\label{roha}
	E(\tS_nf)\,=\,
	\sum_{1\leq i\leq n}
	\alpha_i
	E(X_if)
\,,
\end{equation*}
whence, dividing by $P(f=1)$, 
\begin{equation}
	\label{droha}
		 E\big(\tS_n\,\big|\,f=1\big)
	\,=\,
	\sum_{1\leq i\leq n}
	\alpha_i
	E\big(X_if \,\big|\,f=1\big)
\,.
\end{equation}
The time has come to choose the 
$\alpha_i$'s. We take 
\begin{equation*}
	\label{tcoeffa}
	\forall i\in\unn\qquad
	\alpha_i\,=\,
      \frac{
	E\big(X_if \,\big|\,f=1\big)
      }{
      \sqrt{\displaystyle\sum_{1\leq j\leq n}\Big(E\big(X_jf \,\big|\,f=1\big)\Big)^2
}}
		\,.
\end{equation*}
With this choice, the equation~\eqref{droha} together with the inequality~\eqref{tzfci}
yield
the inequality~\eqref{tmyin} of the proposition.
\end{proof}
\noindent
Naturally, 
	proposition~\ref{rth} is stronger than
	proposition~\ref{bth}: a routine application of the
	Cauchy--Schwarz inequality shows that 
	the inequality~\eqref{vyin} 
	is implied by the inequality~\eqref{tmyin}.
	In the specific case where the function $f$ is non--decreasing,
	we have,
	thanks to formulas~\eqref{rmodif} and~\eqref{rw1}, 
	\begin{equation}
		E\big(X_if \,\big|\,f=1\big)\,=\,
      \frac{
1	
      }{p}
		P\big(i\in \cP(f) \,\big|\,f=1\big)\,,
	\end{equation}
	so that proposition~\ref{rth} yields the
	following corollary, which is an improvement of corollary~\ref{imme}.
	\begin{corollary}
	\label{crth}
	For any  non--decreasing boolean function $f$,
	we have
	\begin{equation}
		\label{gmyin}
	 P(f=1)\,\leq\,
2\exp\Big(
-\frac{ 1}{2}
	\sum_{1\leq i\leq n}
	\Big(
		(1-p)	P\big(i\in \cP(f) \,\big|\,f=1\big)
\Big)^2\Big)\,.
\end{equation}
\end{corollary}
%
\noindent
At this point, we have generalized
to the 
Bernoulli product
measure of parameter~$p$ an inequality that Talagrand proved
in the case where $p=1/2$.
Indeed, 
denoting by $\mu= P_{1/2}$ 
the uniform measure on $\{0,1\}^n$, Talagrand's
inequality can be stated as follows: there exists a universal constant $K$ such that, 
for any  boolean function $f$,
we have
	\begin{equation}
		\label{talag}
	\sum_{1\leq i\leq n}
	\Big( 	
		\int fX_i\,d\mu
\Big)^2
\,\leq\,K \mu(f=1)^2\ln \frac{e}{
\mu(f=1)}\,.
\end{equation}
Clearly, we can rewrite the inequality~\eqref{talag} in an exponential
form to get an inequality similar to~\eqref{tmyin}, albeit
with non explicit constants
(remember that
Talagrand did not care
about their values).
In fact, the main purpose of Talagrand 
was to improve this first inequality (stated in
Proposition~$2.2$ of~\cite{INCT}). This first inequality
was just
a little warm--up before starting the real work! 
\section{The magic of Talagrand}
So, how did Talagrand proceed to go further?
Noticing that proposition~\ref{rth} holds
for any boolean function,
he applied it to the indicator function of the event
$\{\,k\in\ \cP(f)\,\}$, 
where $f$ is a non--decreasing boolean function. 
This yields
\begin{equation*}
		\label{magi}
	 P\big(
\,k\in\ \cP(f)\,
	 \big)\,\leq\,
2\exp\Big(
-\frac{ 1}{2}
	\sum_{1\leq i\leq n}
	\Big(
		p(1-p)	
	E\big(X_i 1_{\{ \,k\in\ \cP(f)\,\}}
		 \,\big|
		 \,k\in\ \cP(f)\,
	 \big)
\Big)^2\Big)\,.
\end{equation*}
The variable $X_k$ and the event
$\{ \,k\in\ \cP(f)\,\}$ are independent, therefore
\begin{equation*}
	E\big(X_k 1_{\{ \,k\in\ \cP(f)\,\}}\big)\,=\,0\,.
	\end{equation*}
Using formula~\eqref{rmodif} in the reverse direction, we can rewrite
the expectation 
for $i\neq k$
as follows:
\begin{equation}
	\forall i\neq k\qquad
	E\big(X_i 1_{\{ \,k\in\ \cP(f)\,\}}\big)\,=\,
	E\big(X_i X_kf\,\big)\,,
	\end{equation}
and the previous inequality becomes, in a non exponential form similar
to Talagrand's first inequality~\eqref{talag},
\begin{equation*}
		\label{tagi}
\sum_{\genfrac{}{}{0pt}{2}{1\leq i\leq n}{i\neq k}}
	\Big(
	E\big(X_i X_kf\,
	 \big)
\Big)^2
	 \,\leq\,
\frac{2 }{p^2(1-p)^2}	
	 \Big(P\big( k\in\ \cP(f)\big)\Big)^2
	 \ln \Big(
\frac{2}{
	 P\big(
\,k\in\ \cP(f)\,
\big)}\Big)
\,.
\end{equation*}
Next, 
we note
$c(p)={2 }{p^{-2}(1-p)^{-2}}$	
and we sum this inequality over $k$ to get
\begin{equation*}
		\label{stagi}
\sum_{\genfrac{}{}{0pt}{2}{1\leq i,k\leq n}{i\neq k}}
	\Big(
	E\big(X_i X_kf\,
	 \big)
\Big)^2
	 \,\leq\,c(p)
	\sum_{1\leq k\leq n}\!
	 \Big(P\big( k\in\ \cP(f)\big)\Big)^2
	 \ln \Big(
\frac{2}{
	 P\big(
\,k\in\ \cP(f)\,
\big)}\Big)
\,.
\end{equation*}
Finally,
in one of those feats of strength for which only he
has the secret, Talagrand proved the following
much stronger inequality for the uniform 
measure $\mu= P_{1/2}$: there exists a 
universal constant $K$ such that, 
for any  non--decreasing boolean function $f$,
we have
	\begin{equation}
		\label{etalag}
\sum_{\genfrac{}{}{0pt}{2}{1\leq i,k\leq n}{i\neq k}}
	\Big( 	
		\int fX_iX_k\,d\mu
\Big)^2
\,\leq\,K 
	\sum_{1\leq k\leq n}
	\mu\big( k\in\ \cP(f)\big)^2
\ln \frac{K}{
	\sum_{1\leq k\leq n}
	\mu\big( k\in\ \cP(f)\big)^2
}\,.
\end{equation}
The considerable improvement lies in the presence of the sum inside the
logarithm. 
Talagrand's proof is beautiful and mysterious. It rests
on a complicated 
intermediate result which involves a sort of
bootstrap mechanism.
The inequality~\eqref{etalag} attracted the 
interest of several researchers and
led to numerous unexpected developments.
For instance, the whole
theory of noise sensitivity of 
Benjamini, Kalai and Schramm \cite{BKS}
was born with a
generalization of this inequality.
An important question was whether inequality~\eqref{etalag} holds
for the product measure $P$ for $p\neq 1/2$.
It is known that the inequalities involving the uniform measure
can in principle be generalized to
the Bernoulli product measure. Keller \cite{KER} has designed
a procedure to deduce automatically some inequalities for the
biased measure from inequalities involving the uniform measure.
Anyway, the natural approach consists in adapting the proofs 
written for the uniform measure, as we did in
this paper for the first inequality of Talagrand.
However, this might turn out
to be very difficult in some cases.
After more 
than twenty years, Talagrand's inequality 
has been 
extended to the biased case by Keller and Kindler \cite{KEKI},
with a new proof which is simpler than the original proof of Talagrand.
Even more, they managed to prove the inequality for any boolean
function, non necessarily non--decreasing. However, upon inspection,
it seems that the proof of Talagrand would work also for any boolean
function!

\bibliographystyle{amsplain}
\bibliography{inc}
\end{document}